\documentclass[a4paper,10pt]{article}
\usepackage{amsmath,amsthm,xypic,
amssymb,amsfonts,
pdfsync,stmaryrd,mathrsfs,yfonts
}
\newtheorem{thm}{Theorem}[section]

\newtheorem{prop}[thm]{Proposition}
\newtheorem{cor}[thm]{Corollary}

\newtheorem{lem}[thm]{Lemma}
\theoremstyle{definition}

\newtheorem{defi}[thm]{Definition}

\newtheorem{rem}[thm]{Remark}

\newtheorem{notation}[thm]{Notation}
\newtheorem{exm}[thm]{Example}

\theoremstyle{plain}

\newcommand{\lla}{\langle\!\langle}
\newcommand{\rra}{\rangle\!\rangle}

\renewcommand{\phi}{\varphi}

\newcommand{\alt}{\mathop{\mathrm{Alt}}}
\newcommand{\co}{\mathop{\mathrm{cor}}}

\newcommand{\sym}{\mathop{\mathrm{Sym}}}

\newcommand{\car}{\mathop{\mathrm{char}}}
\newcommand{\id}{\mathop{\mathrm{id}}}

\newcommand{\disc}{\mathop{\mathrm{disc}}}

\newcommand\iso{\xrightarrow{
   \,\smash{\raisebox{-0.3ex}{\ensuremath{\scriptstyle\sim}}}\,}}

\author{A.-H. Nokhodkar}

\title{Quadratic descent of totally decomposable orthogonal involutions in characteristic two}

\begin{document}
\maketitle

\begin{abstract}
We investigate the quadratic descent of totally decomposable algebras with involution of orthogonal type in characteristic two.
Both separable and inseparable extensions are included.
\end{abstract}

\section{Introduction}

Let $K/F$ be a field extension and let $(A,\sigma)$ be a central simple algebra with involution over $K$.
By a {\it descent} of $(A,\sigma)$ to $F$ we mean an $F$-algebra with involution $(B,\sigma')$ satisfying $(A,\sigma)\simeq_K(B,\sigma')\otimes_F(K,\id)$.
In the case where $\car F\neq2$ and $K/F$ is a quadratic extension, the descent of orthogonal involutions of degree $2$ and $4$ over $K$ was studied  in \cite{dherte} and necessary and sufficient
conditions in terms of the discriminant and the Clifford algebra were given.
Also, if $\car F=2$ and $K/F$ is a finite extension satisfying $K^2\subseteq F$, a criterion for totally decomposable algebras with orthogonal involution over $K$ to have a descent to $F$ was obtained in \cite[(6.2)]{mn2}.

Let $(A,\sigma)$ be a totally decomposable algebra with orthogonal involution over a field $F$ of characteristic $2$.
In \cite{mn2}, a maximal commutative Frobenius subalgebra $\Phi(A,\sigma)$ of $A$ was introduced and the relation between this subalgebra and the {\it Pfister invariant} $\mathfrak{Pf}(A,\sigma)$, a bilinear Pfister form associated to $(A,\sigma)$ defined in \cite{dolphin}, was studied.
It was also shown that totally decomposable orthogonal involutions can be classified, up to conjugation, by their Pfister invariant (see \cite[(6.5)]{mn2}).

In this work we study the quadratic descent of totally decomposable ortho\-gonal involutions in characteristic $2$.
We first consider inseparable extensions.
If $K=F(\sqrt\alpha)$ is a quadratic extension of $F$ and $(A,\sigma)$ is a totally decomposable algebra with orthogonal involution over $K$, we find necessary and sufficient conditions for $(A,\sigma)$,
in terms of $\Phi(A,\sigma)$ and $\mathfrak{Pf}(A,\sigma)$, to has a descent to $F$.
It is also shown that if $(A,\sigma)$ has a descent to $F$, then it has a {\it totally decomposable descent} to $F$, i.e., a descent which is totally decomposable itself (see (\ref{insep})).
This result complements \cite[(6.2)]{mn2} for quadratic extensions.

We next consider separable extensions.
In \S\ref{sec-co} we investigate the quadratic descent of quaternion algebras with orthogonal involution.
In \S\ref{sec-conic} some properties of totally singular conic algebras are studied.
These results will be used in \S\ref{sec-an} to obtain a descent property of the Pfister invariant.
Finally, in \S\ref{sec-sep} we state our main result (\ref{main}) which asserts that for a separable quadratic extension $K/F$ and a totally decomposable algebra with orthogonal involution $(A,\sigma)$ over $K$, the following conditions are equivalent:
(1) $(A,\sigma)$ has a descent to $F$. (2) $(A,\sigma)$ has a totally decomposable descent to $F$.
(3) The corestriction $\co_{K/F}(A)$ splits and $\mathfrak{Pf}(A,\sigma)\simeq\mathfrak{b}_K$ for some symmetric bilinear form $\mathfrak{b}$ over~$F$.
In the case where $\sigma$ is isotropic, these conditions are also equivalent to: (4) $\co_{K/F}(A)$ splits and $\Phi(A,\sigma)\simeq S\otimes_FK$ for some totally singular conic $F$-algebra $S$.

\section{Preliminaries}
Throughout this work, all fields are supposed to be of characteristic $2$.

Let $A$ be a central simple algebra over a field $F$.
An {\it involution} on $A$ is a map $\sigma:A\rightarrow A$ satisfying $\sigma^2(x)=x$, $\sigma(x+y)=\sigma(x)+\sigma(y)$ and $\sigma(xy)=\sigma(y)\sigma(x)$ for $x,y\in A$.
For an algebra with involution $(A,\sigma)$ we define the subspaces
\[\alt(A,\sigma)=\{x-\sigma(x)\mid x\in A\}\quad {\rm and}\quad \sym(A,\sigma)=\{x\in A\mid\sigma(x)=x\}.\]
If $\sigma|_F=\id$, we say that $\sigma$ is of {\it the first kind}.
Otherwise, $\sigma$ is said to be of {\it the second kind}.
An involution $\sigma$ of the first kind is called {\it symplectic} if after scalar extension to a splitting field, it becomes adjoint to an alternating bilinear form.
Otherwise, it is called {\it orthogonal}.
According to \cite[(2.6 (2)]{knus}, $\sigma$ is orthogonal if and only if $1\notin\alt(A,\sigma)$.
We denote the discriminant of an orthogonal involution $\sigma$ by $\disc\sigma$ (see \cite[(7.2)]{knus}).
An algebra with involution $(A,\sigma)$ (or the involution $\sigma$ itself) is called {\it isotropic} if there exists a nonzero element $x\in A$ such that $\sigma(x)x=0$.
Otherwise, it is called {\it anisotropic}.

A {\it quaternion algebra} over a field $F$ is a four-dimensional central simple $F$-algebra.
Every quaternion algebra $Q$ has a basis $(1,u,v,w)$, called a {\it quaternion basis}, satisfying
$u^2+u\in F$, $v^2\in F^{\times}$ and $w=uv=vu+v$ (see \cite[p. 25]{knus}).
In this case, if $a=u^2+u\in F$ and $b=v^2\in F^\times$, then $Q$ is denoted by $[a,b)_F$.
If $b\in F^{\times2}$ or $a\in\wp(F):=\{x^2+x\mid x\in F\}$, then $Q$ splits.
Also, it is readily verified that every element $v'\in Q\setminus F$ satisfying $v^{\prime2}\in F^\times$ extends to a quaternion basis $(1,u',v',w')$ of $Q$.

Let $V$ be a finite-dimensional vector space over a field $F$.
A symmetric bilinear form $\mathfrak{b}:V\times V\rightarrow F$ is called {\it metabolic} if there exists a subspace $W$ of $V$ with $\dim_FW=\frac{1}{2}\dim_FV$ such that $\mathfrak{b}|_{W\times W}=0$.
We say that two bilinear forms $\mathfrak{b}$ and $\mathfrak{b}'$ are {\it similar} if $\mathfrak{b}\simeq\lambda\mathfrak{b}'$ for some $\lambda\in F^\times$.
For $\alpha\in F$ the isometry class of the symmetric bilinear form $\mathfrak{b}((x_1,x_2),(y_1,y_2))=x_1y_1+\alpha x_2y_2$ on $F^2$ is denoted by $\langle 1,\alpha\rangle$.
Also, for $\alpha_1,\cdots,\alpha_n\in F^\times$, the form $\langle 1,\alpha_1\rangle\otimes\cdots\otimes\langle 1,\alpha_n\rangle$ is called a {\it bilinear ($n$-fold) Pfister form} and is denoted by $\lla \alpha_1,\cdots,\alpha_n\rra$.
Finally, if $K/F$ is a field extension and $\mathfrak{b}$ is a bilinear form over $F$, the scalar extension of $\mathfrak{b}$ to $K$ is denoted by $\mathfrak{b}_K$.

\section{Inseparable descent}\label{sec-ins}
We begin our discussion with a definition from \cite{mn2}.
\begin{defi}
An algebra $A$ over a field $F$ is called a {\it totally singular conic algebra} if $x^2\in F$ for every $x\in R$.
\end{defi}
According to \cite[(3.2)]{mn2}, every totally singular conic algebra is a local commutative algebra.

\begin{defi}
An algebra with involution $(A,\sigma)$ over a field $F$ is called {\it totally decomposable} if it is isomorphic to a tensor product of quaternion $F$-algebras with involution.
\end{defi}
Note that if $\sigma$ is orthogonal and $(A,\sigma)\simeq\bigotimes_{i=1}^n(Q_i,\sigma_i)$ is a decomposition of $(A,\sigma)$ into quaternion $F$-algebras with involution, then every $\sigma_i$ is necessarily orthogonal by \cite[(2.21)]{knus}.

Let $(A,\sigma)$ be a totally decomposable algebra of degree $2^n$ with orthogonal invo\-lution over $F$.
In \cite{mn2}, it was shown that there exists a $2^n$-dimensional totally singular conic $F$-algebra $\Phi(A,\sigma)\subseteq F\oplus\alt(A,\sigma)$ which is maximal commutative (i.e., $C_A(\Phi(A,\sigma))=\Phi(A,\sigma)$) and is generated, as an $F$-algebra, by $n$ elements.
Also, according to \cite[(5.10)]{mn2} the subalgebra $\Phi(A,\sigma)$ is uniquely determined, up to isomorphism.
By \cite[(4.6)]{mn2} we have $\Phi(A,\sigma)=F[v_1,\cdots,v_n]$ for some $v_1,\cdots,v_n\in\alt(A,\sigma)$.

\begin{lem}\label{new}
Let $K/F$ be a finite field extension satisfying $K^2\subseteq F$ and let $(Q,\sigma)$ be a quaternion algebra with orthogonal involution over $K$.
For every $v\in\alt(Q,\sigma)$ satisfying $v^2\in F^\times$, there exists a descent $(Q_0,\sigma_0)$ of $(Q,\sigma)$ to $F$ such that $v\in\alt(Q_0,\sigma_0)$.
\end{lem}

\begin{proof}
See \cite[(6.1)]{mn2}.
\end{proof}

\begin{lem}\label{insep}
Let $K/F$ be a finite field extension satisfying $K^2\subseteq F$ and let $(A,\sigma)$ be a totally decomposable algebra  with orthogonal involution over $K$.
Then $(A,\sigma)$ has a totally decomposable descent to $F$ if and only if $x^2\in F$ for every $x\in \Phi(A,\sigma)$.
\end{lem}

\begin{proof}
If $x^2\in F$ for every $x\in\Phi(A,\sigma)$, then the proof of \cite[(6.2)]{mn2} shows that $(A,\sigma)$ has a totally decomposable descent to $F$.
Conversely, if $(B,\sigma')$ is a totally decomposable descent of $(A,\sigma)$ to $F$, then $\Phi(A,\sigma)\simeq\Phi(B,\sigma')\otimes_FK$ as $K$-algebras.
Since $K^2\subseteq F$ and $v^2\in F$ for every $v\in\Phi(B,\sigma')$, we have $x^2\in F$ for every $x\in \Phi(A,\sigma)$.
\end{proof}

\begin{defi}
Let $(A,\sigma)\simeq\bigotimes_{i=1}^n(Q_i,\sigma_i)$ be a totally decomposable algebra with orthogonal involution over a field $F$ and let $\alpha_i\in F^\times$ be a representative of the class $\disc\sigma_i\in F^\times/F^{\times2}$.
As in \cite{dolphin}, we call the bilinear Pfister form $\lla\alpha_1,\cdots,\alpha_n\rra$, the {\it Pfister invariant} of $(A,\sigma)$ and we denote it by $\mathfrak{Pf}(A,\sigma)$.
\end{defi}
Note that according to \cite[(7.5)]{dolphin}, the isometry class of $\mathfrak{Pf}(A,\sigma)$ does not depend on the decomposition of $(A,\sigma)$.
Also, as observed in \cite[(5.5)]{mn2}, the algebra $\Phi(A,\sigma)$ can be considered as an underlying vector space of $\mathfrak{Pf}(A,\sigma)$ in such a way that for $v,w\in\Phi(A,\sigma)$, $\mathfrak{Pf}(A,\sigma)(v,w)$ is the unique element $\alpha\in F$ satisfying $vw+\alpha\in \alt(A,\sigma)$.
In particular, $\mathfrak{Pf}(A,\sigma)(v,v)=v^2$ for every $v\in\Phi(A,\sigma)$.

The next result complements \cite[(6.2)]{mn2} for quadratic extensions.
\begin{thm}
Let $K=F(\sqrt{\alpha})$ be a quadratic extension of a field $F$.
For a totally decomposable algebra with orthogonal involution $(A,\sigma)$ over $K$ the following conditions are equivalent:
\begin{itemize}
  \item [(1)] $(A,\sigma)$ has a descent to $F$.
  \item [(2)] $(A,\sigma)$ has a totally decomposable descent to $F$.
  \item [(3)] $x^2\in F$ for every $x\in\Phi(A,\sigma)$.
  \item[(4)] $\mathfrak{Pf}(A,\sigma)\simeq\mathfrak{b}_K$ for some symmetric bilinear form $\mathfrak{b}$ over $F$.
\end{itemize}
\end{thm}

\begin{proof}
The equivalence of $(2)$ and $(3)$ follows from (\ref{insep}).
The implication $(2)\Rightarrow(1)$ is evident and $(2)\Rightarrow(4)$ follows by setting $\mathfrak{b}=\mathfrak{Pf}(C,\tau)$, where $(C,\tau)$ is a totally decomposable descent of $(A,\sigma)$.
It suffices therefore to prove the implications $(1)\Rightarrow(3)$ and $(4)\Rightarrow(3)$.

$(1)\Rightarrow(3)$:
Let $(B,\sigma')$ be a descent of $(A,\sigma)$ and let
\[g:(A,\sigma)\iso(B,\sigma')\otimes_F(K,\id),\]
be an isomorphism of $K$-algebras with involution.
Write $\Phi(A,\sigma)=F[v_1,\cdots,v_n]$ for some $v_1,\cdots,v_n\in \alt(A,\sigma)$.
Since $g(\alt(A,\sigma))=\alt(B,\sigma')\otimes_FK$, there exist $u_i,w_i\in\alt(B,\sigma')$ such that $g(v_i)=u_i\otimes1+w_i\otimes\sqrt\alpha$, $i=1,\cdots,n$.
Hence
\[g(v_i)^2=(u_i^2+\alpha w_i^2)\otimes1+(u_iw_i+w_iu_i)\otimes\sqrt{\alpha}.\]
Since $v_i^2\in K$, the element $g(v_i)^2$ lies in the center of $B\otimes_FK$, i.e., $u_i^2+\alpha w_i^2\in F$ and $u_iw_i+w_iu_i\in F$.
We have $u_iw_i+w_iu_i=u_iw_i-\sigma(u_iw_i)\in\alt(B,\sigma')$.
The orthogonality of $\sigma'$ implies that $1\notin\alt(B,\sigma')$, hence $u_iw_i+w_iu_i=0$.
It follows that $g(v_i)^2=(u_i^2+\alpha w_i^2)\otimes1\in F\otimes_FF$, i.e., $v_i^2\in F$.
Since $K^2\subseteq F$ and $\Phi(A,\sigma)$ is commutative, we get $x^2\in F$ for every $x\in\Phi(A,\sigma)$.

$(4)\Rightarrow(3)$:
Let $V$ be an underlying vector space of $\mathfrak{b}$ and consider an isometry \[h:(\Phi(A,\sigma),\mathfrak{Pf}(A,\sigma))\iso(V\otimes_FK,\mathfrak{b}_K).\]
Let $v\in\Phi(A,\sigma)$ and write $h(v)=u\otimes1+w\otimes\sqrt\alpha$ for some $u,w\in V$.
Then
\begin{align*}
v^2&=\mathfrak{Pf}(A,\sigma)(v,v)=\mathfrak{b}_K(h(v),h(v))=\mathfrak{b}_K(u\otimes1+w\otimes\sqrt\alpha,u\otimes1+w\otimes\sqrt\alpha)\\
&=\mathfrak{b}(u,u)+\sqrt{\alpha}\mathfrak{b}(u,w)+\sqrt{\alpha}\mathfrak{b}(w,u)+\alpha\mathfrak{b}(w,w)
=\mathfrak{b}(u,u)+\alpha\mathfrak{b}(w,w)\in F,
\end{align*}
which completes the proof.
\end{proof}

\section{Separable descent of quaternion algebras}\label{sec-co}
From now on we fix a separable quadratic extension $K/F$ and two elements $\eta\in K$ and $\delta\in F\setminus \wp(F)$ such that $K=F(\eta)$ and $\eta^2+\eta=\delta$.
The {\it trace} of an element $x\in K$ over $F$ is denoted by $T_{K/F}(x)$.
Hence $T_{K/F}(a+b\eta)=b$ for $a,b\in F$.

Let $A$ be a central simple algebra over $K$ and let $\iota$ be the nontrivial automorphism of $K/F$.
Define the {\it conjugate} algebra $^\iota A=\{^\iota x\mid x\in A\}$ with the operations
\[^\iota x+{^\iota y}= {^\iota}(x+y),\quad ^\iota x^\iota y={^\iota}(xy)\quad {\rm and}\quad ^\iota(\alpha x)=\iota(\alpha){^\iota x},\]
for $x,y\in A$ and $\alpha\in K$.
Let $s:{^\iota}A\otimes_KA\rightarrow{^\iota}A\otimes_KA$ be the switch map induced by $s({^\iota}x\otimes y)={^\iota} y\otimes x$ for $x,y\in A$.
The {\it corestriction} of $A$ is defined as follows:
\[{\co}_{K/F}(A)=\{u\in{^\iota}A\otimes_KA\mid s(u)=u\}.\]
According to \cite[(3.13 (4))]{knus}, $\co_{K/F}(A)$ is a central simple algebra over $F$.

\begin{lem}\label{sc}
A quaternion $K$-algebra $Q$ has a descent to $F$ if and only if $\co_{K/F}(Q)$ splits.
\end{lem}

\begin{proof}
The result follows from \cite[(2.22)]{knus} and \cite[Ch. 8, (9.5)]{schar}.
\end{proof}

The following result can be found in \cite[Ch. I, Exercise 9]{knus}.

\begin{lem}\label{knus}
If $a\in K$ and $b\in F^\times$, then
$\co_{K/F}([a,b)_K)\sim[T_{K/F}(a),b)_F$, where $\sim$ denotes Brauer-equivalence.
\end{lem}

\begin{lem}\label{wp}
Let $Q=[a,b)_E$ be a quaternion algebra over a field $E$.
If $b\notin E^{\times2}$, then $Q$ splits if and only if $a+c^2 b\in\wp(E)$ for some $c\in E$.
\end{lem}

\begin{proof}
Consider the quadratic \'etale extension $E_a:=E[X]/(X^2+X+a)$ of $E$.
By \cite[(98.14) (5)]{elman}, $Q$ splits if and only if $b$ is a norm in $E_a$, or equivalently
$b=\beta^2+\beta\lambda+\lambda^2a$ for some $\beta,\lambda\in E$.
As $b\notin E^{\times2}$, we have $\lambda\neq0$.
Hence $Q$ splits if and only if $a+\lambda^{-2}b=\lambda^{-1}\beta+(\lambda^{-1}\beta)^2\in\wp(E)$.
\end{proof}

\begin{cor}\label{cor}
Let $[a,b)_E$ and $[c,b)_E$ be quaternion algebras over a field $E$.
If $b\notin E^{\times2}$, then $[a,b)_E\simeq[c,b)_E$ if and only if $a+c+d^2b\in\wp(E)$ for some $d\in E$.
\end{cor}

\begin{proof}
The result follows from (\ref{wp}) and \cite[Ch. 8, (11.1)]{schar}.
\end{proof}

\begin{lem}\label{usym}
Let $(Q,\sigma)$ be a quaternion algebra with orthogonal involution over a field $E$ and let $(1,u,v,w)$  be a quaternion basis of $Q$.
If $v\in\alt(Q,\sigma)$, then $\sigma(u)=u$.
\end{lem}

\begin{proof}
Since $\sigma(u)-u\in\alt(Q,\sigma)$ and $\dim_E\alt(Q,\sigma)=1$, there exists $a\in E$ such that $\sigma(u)-u=a v$.
Hence $\sigma(uv)=v(u+a v)=uv+v+a b$, where $b=v^2\in F^{\times}$.
It follows that $v+ab\in\alt(Q,\sigma)$, i.e., $ab\in\alt(Q,\sigma)$.
As $\sigma$ is orthogonal, we get $a=0$.
\end{proof}

The next result is easily deduced from \cite[(7.4)]{knus}.
\begin{lem}\label{disc}
Let $E$ be a field and let $t$ be the transpose involution on $M_2(E)$.
For a quaternion algebra with orthogonal involution $(Q,\sigma)$ over $E$ we have $(Q,\sigma)\simeq(M_2(E),t)$ if and only if $\disc\sigma$ is trivial.
\end{lem}

\begin{prop}\label{quaternion}(Compare \cite[(2.4)]{dherte})
A quaternion algebra with orthogonal involution $(Q,\sigma)$ over $K$ has a descent to $F$ if and only if
$\co_{K/F}(Q)$ splits and there exists $v\in\alt(Q,\sigma)$ such that $v^2\in F^\times$.
In addition, if $(Q_0,\sigma_0)$ is a descent of $(Q,\sigma)$ to $F$, then $v\in\alt(Q_0,\sigma_0)$.
\end{prop}

\begin{proof}
If $Q\simeq_K Q_0\otimes_F K$ for some $\sigma$-invariant quaternion $F$-algebra $Q_0\subseteq Q$, then $\co_{K/F}(Q)$ splits by (\ref{sc}).
Also for every unit $v\in\alt(Q_0,\sigma|_{Q_0})$ we have $v^2\in F^\times$.
Conversely, let $b=v^2\in F^\times$.
If $b\in F^{\times2}$, then by (\ref{disc}) we have
\[(Q,\sigma)\simeq_K(M_2(K),t)\simeq_K(M_2(F),t)\otimes_F(K,\id),\] and we are done.
Suppose that $b\notin F^{\times2}$.
Extend $v$ to a quaternion basis $(1,u,v,w)$ of $Q$ and set $a=u^2+u\in K$.
Then $Q\simeq[a,b)_K$.
Write $a=\beta+\lambda\eta$ for some $\beta,\lambda\in F$.
By (\ref{knus}), $\co_{K/F}(Q)\sim[T_{K/F}(a),b)_F$, thus $[\lambda,b)_F$ splits.
By (\ref{wp}) there exist $d,e\in F$ such that $\lambda=d^2+d+e^2b$.
Set $u'=u+e\eta v+d\eta\in Q$ and $c=\beta+d^2\delta+e^2\delta b\in F$.
Then $u'v+vu'=v$ and $u^{\prime2}+u'=c$.
Hence the $F$-algebra $Q_0$ generated by $u'$ and $v$ is a quaternion algebra, isomorphic to $[c,b)_F$.
We also have
\begin{align*}
a+c&=(\beta+\lambda\eta)+(\beta+d^2\delta+e^2\delta b)=(d^2+d+e^2b)\eta+d^2\delta+e^2\delta b \\
&=d^2\eta+d^2\delta+d\eta+e^2\eta b+e^2\delta b=(d\eta)^2+d\eta+(e\eta)^2b.
\end{align*}
Hence $a+c+(e\eta)^2b\in\wp(K)$.
Using (\ref{cor}) we get $[c,b)_K\simeq[a,b)_K$, i.e., $Q\simeq_K Q_0\otimes_FK$.
Note that $(1,u',v,u'v)$ is a quaternion basis of $[c,b)_K$, hence $\sigma(u')=u'$ by (\ref{usym}).
It follows that $\sigma(Q_0)=Q_0$, i.e., $(Q,\sigma)\simeq_K(Q_0,\sigma|_{Q_0})\otimes_F(K,\id)$.

To prove the last assertion of the result, observe that as $\alt(Q,\sigma)$ is one-dimensional, for every unit $w\in\alt(Q_0,\sigma_0)$ we have $\alt(Q,\sigma)=Kw$.
Hence one can write $v=\alpha w$ for some $\alpha\in K^\times$.
Since $v^2,w^2\in F^\times$ and $K^{\times2}\cap F^\times=F^{\times2}$, we get $\alpha\in F^\times$, hence $v\in\alt(Q_0,\sigma_0)$.
\end{proof}

\section{Some descent properties of totally singular conic algebras}\label{sec-conic}
We recall that $K=F(\eta)$ is a separable quadratic extension of $F$ with $\eta^2+\eta=\delta\in F\setminus\wp(F)$.

\begin{lem}\label{trace}
Let $(A,\sigma)$ be a totally decomposable algebra with orthogonal invo\-lution over $K$.
Suppose that $\mathfrak{Pf}(A,\sigma)\simeq\lla\alpha_1,\cdots,\alpha_n\rra$ for some $\alpha_1,\cdots,\alpha_n\in K^\times$.
If $(A,\sigma)$ has a descent to $F$, then $\mathfrak{Pf}(A,\sigma)\simeq\lla\iota(\alpha_1),\cdots,\iota(\alpha_n)\rra$, where $\iota$ is the nontrivial automorphism of $K/F$.
\end{lem}

\begin{proof}
Let $(B,\sigma')$ be a descent of $(A,\sigma)$ to $F$ and consider an isomorphism of $K$-algebras with involution
\[f:(A,\sigma)\iso(B,\sigma')\otimes_F(K,\id).\]
For $x\in B$ and $\alpha\in K$, the assignment $x\otimes\alpha\mapsto x\otimes\iota(\alpha)$ induces an $F$-algebra isomorphism $g:B\otimes_FK\rightarrow B\otimes_FK$ satisfying $g\circ(\sigma'\otimes\id)=(\sigma'\otimes\id)\circ g$.
The composition
\[h=f^{-1}\circ g\circ f:A\rightarrow A,\]
is therefore an $F$-algebra isomorphism satisfying $h\circ\sigma=\sigma\circ h$.
By \cite[(3.1)]{me}, there exists a decomposition $(A,\sigma)\simeq\bigotimes_{i=1}^n(Q_i,\sigma_i)$ into quaternion $K$-algebras with orthogonal involution such that
\[\disc\sigma_i=\alpha_iK^{\times2}\in\penalty 0 K^\times/K^{\times2},\quad i=1,\cdots,n.\]
Set $Q'_i=h(Q_i)$, $i=1,\cdots,n$.
We claim that every $Q'_i$ is a $\sigma$-invariant quaternion $K$-subalgebra of $A$.
Let $(1,u_i,v_i,w_i)$ be a quaternion basis of $Q_i$ and set $u'_i=h(u_i)$, $v'_i=h(v_i)$ and $w'_i=h(w_i)$.
Since $h(\alpha)=\iota(\alpha)\in K$ for every $\alpha\in K$, the condition $u_i^2+u_i\in K$ implies that $u_i^{\prime2}+u'_i\in K$.
Similarly, as $h(\alpha)=\iota(\alpha)\in K^\times$ for every $\alpha\in K^\times$, we have $v_i^{\prime2}\in K^\times$.
The equality $u_iv_i+v_iu_i=v_i$ implies that $u'_iv'_i+v'_iu'_i=v'_i$.
Hence, $Q_i'$ is a quaternion $K$-algebra with a quaternion basis $(1,u'_i,v'_i,w'_i)$ for $i=1,\cdots,n$.
We also have
\[\sigma(Q'_i)=\sigma(h(Q_i))=h(\sigma(Q_i))=h(Q_i)=Q'_i,\]
proving the claim.
Now, choose $x_i\in\alt(Q_i,\sigma_i)$ with $x_i^2=\alpha_i$, $i=1,\cdots,n$.
Then $h(x_i)\in\alt(Q'_i,\sigma|_{Q'_i})$ satisfies $h(x_i)^2=\iota(\alpha_i)$, hence $\disc\sigma|_{Q'_i}=\iota(\alpha_i)K^{\times2}$.
We also have $(A,\sigma)\simeq\bigotimes_{i=1}^n(Q'_i,\sigma|_{Q'_i})$, which implies in particular that
$\mathfrak{Pf}(A,\sigma)\penalty 0\simeq\lla\iota(\alpha_1),\cdots,\iota(\alpha_n)\rra$.
\end{proof}

\begin{prop}\label{an}
Let $(A,\sigma)$ be a totally decomposable algebra with orthogonal involution over $K$.
If $(A,\sigma)$ has a descent to $F$, then there exists a totally singular conic $F$-algebra $S\subseteq\Phi(A,\sigma)$ for which $\Phi(A,\sigma)\simeq S\otimes_FK$ as $K$-algebras.
\end{prop}

\begin{proof}
Let $(A,\sigma)\simeq\bigotimes_{i=1}^n(Q_i,\sigma_i)$ be a decomposition of $(A,\sigma)$ into quaternion $K$-algebras with involution.
Let $v_i\in\alt(Q_i,\sigma_i)$ be a unit and set $\alpha_i=v_i^2\in K^\times$, $i=1,\cdots,n$.
Then $\mathfrak{Pf}(A,\sigma)\simeq\lla\alpha_1,\cdots,\alpha_n\rra$ and $\Phi(A,\sigma)\simeq K[v_1,\cdots,v_n]$.
Write
$\alpha_i=b_i+c_i\eta$ for some $b_i,c_i\in F$, $i=1,\cdots,n$.
By (\ref{trace}) we have
\[\mathfrak{Pf}(A,\sigma)\simeq\lla b_1+c_1 +c_1\eta,\cdots,b_n+c_n+c_n\eta\rra.\]
By \cite[(5.6)]{mn2}, there exists $v'_i\in \alt(A,\sigma)$ such that $v_i^{\prime2}=b_i+c_i+ c_i\eta$, $i=1,\cdots,n$, and $\Phi(A,\sigma)\simeq K[v'_1,\cdots,v'_n]$.
For $i=1,\cdots,n$, we may identify $v_i$ (resp. $v'_i$) with an element of $\Phi(A,\sigma)$.
Set
\begin{align*}
u_i=v_i+v'_i\in\Phi(A,\sigma)\quad {\rm and}\quad
w_i=(1+\eta)v_i+\eta v'_i\in\Phi(A,\sigma),
\end{align*}
so that
$u_i^2=c_i\in F$ and $w_i^2=b_i+c_i\delta\in F$, $i=1,\cdots,n$.
Set $S=F[u_1,\cdots,u_n,w_1,\cdots,w_{n}]$.
Then $S$ is a totally singular conic $F$-algebra and
\[S\otimes_FK\simeq K[u_1,\cdots,u_n,w_1,\cdots,w_{n}]\subseteq \Phi(A,\sigma).\]
On the other hand, as $v_i=\eta u_i+w_i$, the algebra $S\otimes_FK$ contains a copy of $v_i$ for $i=1,\cdots,n$.
Hence $\Phi(A,\sigma)\simeq S\otimes_FK$.
\end{proof}

\begin{rem}\label{m}
According to \cite[(3.2)]{mn2}, every totally singular conic $F$-algebra $S$ is a local algebra and its unique maximal ideal is $\mathfrak{m}=\{x\in S\mid x^2=0\}$.
It follows that the unique maximal ideal of $S\otimes_FK$ is $\mathfrak{m}\otimes_FK$.
In fact, every $x\in S\otimes_FK$ can be written as $x=u\otimes1+v\otimes\eta$ for some $u,v\in S$, so that $x^2=(u^2+\delta v^2)\otimes1+v^2\otimes\eta$.
Hence, $x^2=0$ if and only if $u^2=v^2=0$, or equivalently, $x\in\mathfrak{m}\otimes_FK$.
\end{rem}

\begin{lem}\label{desphi}
Let $R$ be a totally singular conic algebra over $K$.
Suppose that there exists a totally singular conic $F$-algebra $S$ such that $R\simeq S\otimes_FK$.
If $E$ is a maximal subfield of $R$ containing $K$ and $L$ is a maximal subfield of $S$ containing $F$, then $E\simeq L\otimes_FK$.
\end{lem}

\begin{proof}
Let $\mathfrak{m}$ and $\mathfrak{m}'$ be the respective maximal ideals of $S$ and $R$.
By (\ref{m}), $\mathfrak{m}\otimes_FK$ is the unique maximal ideal of $S\otimes_FK$.
Hence, the isomorphism $R\simeq S\otimes_FK$ induces a $K$-algebra isomorphism
\[R/\mathfrak{m}'\simeq(S\otimes_FK)/(\mathfrak{m}\otimes_FK).\]
Note that there exists a natural $K$-algebra isomorphism $(S\otimes_FK)/(\mathfrak{m}\otimes_FK)\simeq S/\mathfrak{m}\otimes_FK$,
hence
$R/\mathfrak{m}'\simeq S/\mathfrak{m}\otimes_FK$.
By \cite[(3.7 (i))]{mn2}, we have $E\simeq R/\mathfrak{m}'$ as $K$-algebras and $L\simeq S/\mathfrak{m}$ as $F$-algebras, hence $E\simeq L\otimes_FK$.
\end{proof}

\begin{lem}\label{ar}
Let $\mathfrak{b}$ be a bilinear Pfister form over a field $E$.
If $\mathfrak{b}$ is isotropic, then there exist a positive integer $s$ and an anisotropic bilinear Pfister form $\mathfrak{c}$ over $E$ such that $\mathfrak{b}\simeq\lla1\rra^s\otimes\mathfrak{c}$, where $\lla1\rra^s$ is the $s$-fold Pfister form $\lla1,\cdots,1\rra$.
Moreover, the integer $s$ is uniquely determined by the isomorphism class of $\mathfrak{b}$.
\end{lem}

\begin{proof}
See \cite[p. 909]{arason}.
\end{proof}

\begin{notation}
We denote the integer $s$ in (\ref{ar}) by $\mathfrak{i}(\mathfrak{b})$.
If $\mathfrak{b}$ is anisotropic, we set $\mathfrak{i}(\mathfrak{b})=0$.
Also, if $(A,\sigma)$ is a totally decomposable algebra with orthogonal involution over $E$, we simply denote $\mathfrak{i}(\mathfrak{Pf}(A,\sigma))$ by $\mathfrak{i}(A,\sigma)$.
\end{notation}
Note that \cite[(7.5)]{dolphin} implies that $(A,\sigma)$ is anisotropic if and only if $\mathfrak{i}(A,\sigma)=0$.
Also, in view of \cite[(3.1)]{me}, \cite[(7.5)]{dolphin} and (\ref{disc}), if $\mathfrak{i}(A,\sigma)=s>0$, then \[(A,\sigma)\simeq(M_{2^s}(E),t)\otimes(B,\sigma'),\]
where $(B,\sigma')$ is a totally decomposable algebra with anisotropic orthogonal involution over $E$.
Hence, $\Phi(A,\sigma)\simeq\Phi(M_{2^s}(E),t)\otimes\Phi(B,\sigma')$ by \cite[(5.13)]{mn2}.
According to \cite[(6.8)]{mn2}, $\Phi(B,\sigma')$ is a field.
Also, \cite[(5.7)]{mn2} implies that $x^2\in E^2$ for every $x\in \Phi(M_{2^s}(E),t)$.
Hence, $\Phi(B,\sigma')$ may be identified with a maximal subfield of $\Phi(A,\sigma)$ containing $E$.
It follows from \cite[(3.7 (i))]{mn2} that every maximal subfield of $\Phi(A,\sigma)$ containing $E$ is isomorphic to $\Phi(B,\sigma')$.
In particular, such a maximal subfield has dimension $2^{n-s}$.

\begin{lem}\label{desph}
Let $(A,\sigma)$ be a totally decomposable algebra of degree $2^n$ with orthogonal involution over $K$ and let $r=n-\mathfrak{i}(A,\sigma)$.
If $\Phi(A,\sigma)\simeq S\otimes_FK$ for some totally singular conic $F$-algebra $S$, then there exist $v_1,\cdots,v_n\in\Phi(A,\sigma)$ such that
\begin{itemize}
\item[(1)] $\Phi(A,\sigma)=K[v_1,\cdots,v_n]$;
\item[(2)] $K[v_1,\cdots,v_r]$ is a maximal subfield of $\Phi(A,\sigma)$  with $v_i^2\in F$ for $i=1,\cdots,r$;
\item[(3)] $v_i^2=1$ for $i=r+1,\cdots,n$.
\end{itemize}
\end{lem}

\begin{proof}
Let $L$ and $E$ be maximal subfields of $S$ and $\Phi(A,\sigma)$ respectively, with $F\subseteq L$ and $K\subseteq E$.
By (\ref{desphi}) we have  $E\simeq L\otimes_FK$ as $K$-algebras.
Also, as already observed, we have $\dim_FL=\dim_KE=2^r$.
Write $L=F[v_1,\cdots,v_r]$ for some $v_1,\cdots,v_r\in L$.
Then $E\simeq K[v_1,\cdots,v_r]$.
Since $\Phi(A,\sigma)$ is generated, as a $K$-algebra, by $n$ elements, by \cite[(3.9)]{mn2} there exist $v_{r+1},\cdots,v_{n}\in\Phi(A,\sigma)$ such that $\Phi(A,\sigma)=K[v_1,\cdots,v_n]$ and $v_i^2=0$, $i=r+1,\cdots,n$.
This proves parts $(1)$ and $(2)$.
The third part follows by replacing $v_i$ with $v_i+1$ for $i=r+1,\cdots,n$.
\end{proof}

\section{Applications to the Pfister invariant}\label{sec-an}
The following notation was used in \cite{arason}.
\begin{notation}
For a bilinear form $\mathfrak{b}$ over a field $E$ we use the notation
$Q(\mathfrak{b})=\{\mathfrak{b}(v,v)\mid v\in V\}$, where $V$ is an underlying vector space of $\mathfrak{b}$.
\end{notation}
Note that $Q(\mathfrak{b})$ is an $E^2$-subspace of $E$.
Also, since every sum of squares in $E$ is a square, we have $Q(\lla1\rra^s\otimes\mathfrak{b})=Q(\mathfrak{b})$ for every positive integer $s$.

\begin{lem}\label{aras}
Let $\mathfrak{b}$ and $\mathfrak{b}'$ be two anisotropic bilinear $n$-fold Pfister forms over a field $E$ and let $s$ be a positive integer.
Then $\lla1\rra^s\otimes\mathfrak{b}\simeq\lla1\rra^s\otimes\mathfrak{b}'$ if and only if $Q(\mathfrak{b})=Q(\mathfrak{b}')$.
\end{lem}

\begin{proof}
See \cite[p. 909]{arason}.
\end{proof}
We continue to assume that $K=F(\eta)$ is a separable quadratic extension of $F$ with $\eta^2+\eta=\delta\in F\setminus\wp(F)$.

\begin{lem}\label{b}
Let $(A,\sigma)$ be a totally decomposable algebra of degree $2^n$ with isotropic orthogonal involution over $K$.
If $\Phi(A,\sigma)\simeq S\otimes_FK$ for some totally singular conic $F$-algebra $S$, then there exists a symmetric bilinear form $\mathfrak{b}$ over $F$ such that $\mathfrak{Pf}(A,\sigma)\simeq\mathfrak{b}_K$.
\end{lem}

\begin{proof}
Set $s=\mathfrak{i}(A,\sigma)>0$ and $r=n-s$.
By (\ref{desph}), one can write $\Phi(A,\sigma)=K[v_1,\cdots,v_n]$ for some $v_1,\cdots,v_n\in\Phi(A,\sigma)$ such that
$K[v_1,\cdots,v_r]$ is a maximal subfield of $\Phi(A,\sigma)$ with $v_i^2\in F^\times$, $i=1,\cdots,r$, and $v_i^2=1$ for $i=r+1,\cdots,n$.
Set $\alpha_i=v_i^2\in F^\times$, $i=1,\cdots,n$.
Let $\mathfrak{b}$ be the bilinear Pfister form $\lla\alpha_1,\cdots,\alpha_n\rra$ over $F$.
Then $\mathfrak{b}\simeq\lla1\rra^s\otimes\mathfrak{c}$, where $\mathfrak{c}=\lla\alpha_1,\cdots,\alpha_r\rra$.
The algebra $F[v_1,\cdots,v_n]$ may be considered as an underlying vector space of $\mathfrak{b}$ such that $\mathfrak{b}(x,x)=x^2$ for $x\in F[v_1,\cdots,v_n]$.
Hence, $\Phi(A,\sigma)\simeq F[v_1,\cdots,v_n]\otimes_FK$ may be considered as an underlying $K$-vector space of $\mathfrak{b}_K$ such that
\begin{equation}\label{eq1}
\mathfrak{b}_K(x,x)=x^2 \quad{\rm for}\quad x\in\Phi(A,\sigma).
\end{equation}
Since $K[v_1,\cdots,v_r]$ is a field, the form $\mathfrak{c}_K$ is anisotropic.
On the other hand, there exists an anisotropic $r$-fold Pfister form $\mathfrak{c}'$ over $K$ for which $\mathfrak{Pf}(A,\sigma)\simeq\lla1\rra^s\otimes\mathfrak{c}'$.
Since $\mathfrak{Pf}(A,\sigma)(x,x)=x^2$ for every $x\in \Phi(A,\sigma)$, (\ref{eq1}) implies that $Q(\mathfrak{b}_K)=Q(\mathfrak{Pf}(A,\sigma))$, i.e., $Q(\mathfrak{c}_K)=Q(\mathfrak{c}')$.
Using (\ref{aras}) we get
\[\mathfrak{Pf}(A,\sigma)\simeq\lla1\rra^s\otimes\mathfrak{c}'\simeq\lla1\rra^s\otimes\mathfrak{c}_K\simeq\mathfrak{b}_K.\qedhere\]
\end{proof}

\begin{lem}\label{sa}
Let $(A,\sigma)$ be a central simple algebra with involution over a field $E$.
For every $x\in\sym(A,\sigma)$ and $y\in\alt(A,\sigma)$ we have $xyx\in\alt(A,\sigma)$.
\end{lem}

\begin{proof}
Write $y=z-\sigma(z)$ for some $z\in A$.
Then
\[xyx=x(z-\sigma(z))x=xzx-\sigma(xzx)\in\alt(A,\sigma).\qedhere\]
\end{proof}

The next result follows from \cite[(6.1)]{dolphin}.
\begin{lem}\label{direct}
Let $(A,\sigma)$ be a totally decomposable algebra with anisotropic ortho\-gonal involution over a field $E$ and let $x\in A$.
If $\sigma(x)x\in\alt(A,\sigma)$, then $x=0$.
\end{lem}

\begin{cor}\label{zero}
Let $(A,\sigma)$ be a totally decomposable algebra with anisotropic ortho\-gonal involution over a field $E$ and let $x\in\sym(A,\sigma)$.
If $x^2+\alpha\in\alt(A,\sigma)$ for some $\alpha\in E$, then $x^2=\alpha$.
\end{cor}

\begin{proof}
By (\ref{sa}) we have $x^4+x^2\alpha=x(x^2+\alpha)x\in\alt(A,\sigma)$.
Hence
\[\sigma(x^2+\alpha)\cdot(x^2+\alpha)=x^4+\alpha^2=x^4+x^2\alpha+\alpha(x^2+\alpha)\in\alt(A,\sigma),\]
which implies that $x^2=\alpha$ by (\ref{direct}).
\end{proof}

\begin{prop}\label{ani}
Let $(A,\sigma)$ be a totally decomposable algebra with orthogonal involution over $K$.
If $(A,\sigma)$ has a descent to $F$, then there exists a symmetric bilinear form $\mathfrak{b}$ over $F$ such that $\mathfrak{Pf}(A,\sigma)\simeq\mathfrak{b}_K$.
\end{prop}

\begin{proof}
If $\sigma$ is isotropic, the result follows from (\ref{an}) and (\ref{b}).
Suppose that $\sigma$ is anisotropic.
Let $(B,\sigma')$ be a descent of $(A,\sigma)$ to $F$ and consider an isomorphism of $K$-algebras with involution
\[g:(A,\sigma)\iso(B,\sigma')\otimes_F(K,\id).\]
Write $\Phi(A,\sigma)=K[v_1,\cdots,v_n]$ for some $v_1,\cdots,v_n\in\alt(A,\sigma)$.
As $g(\alt(A,\sigma))\break=\alt(B,\sigma')\otimes_F K$,
one  can write $g(v_i)=u_i\otimes1+w_i\otimes\eta$ for some $u_i,w_i\in\alt(B,\sigma')$, $i=1,\cdots,n$.
We first show that $u_i^2,w_i^2\in F$ for every $i$.
We have
\begin{align}\label{eq4}
g(v_i)^2=(u_i\otimes1+w_i\otimes\eta)^2=(u_i^2+\delta w_i^2)\otimes1+(u_iw_i+w_iu_i+w_i^2)\otimes\eta.
\end{align}
Set $\alpha_i=u_iw_i+w_iu_i+w_i^2$. Then
\[w_i^2+\alpha_i=u_iw_i+w_iu_i=u_iw_i-\sigma'(u_iw_i)\in\alt(B,\sigma').\]
Since $g(v_i)^2\in K$ we have $\alpha_i\in F$.
As $\sigma$ is anisotropic, $\sigma'$ is also anisotropic.
Hence (\ref{zero}) implies that
\begin{equation}\label{eq5}
w_i^2=\alpha_i\in F,
\end{equation}
 i.e., $u_iw_i=w_iu_i$.
Since $g(v_i)^2\in K$, (\ref{eq4}) and (\ref{eq5}) imply that $u_i^2\in F$ for every $i$.

We now show that $w_iw_j=w_jw_i$ for $i,j=1,\cdots,n$.
As $(w_i+w_j)^2=\alpha_i+\alpha_j+w_iw_j+w_jw_i$, we have $(w_i+w_j)^2+\alpha_i+\alpha_j\in\alt(B,\sigma')$.
Again, (\ref{zero}) implies that $(w_i+w_j)^2=\alpha_i+\alpha_j$, i.e., $w_iw_j=w_jw_i$.
A similar argument shows that $w_iu_j=u_jw_i$ and $u_iu_j=u_ju_i$ for $i,j=1,\cdots,n$.

Set $S=F[u_1,\cdots,u_n,w_1,\cdots,w_{n}]$.
Then $S$ is a totally singular conic $F$-subalgebra of $B$.
Since $\Phi(A,\sigma)=K[v_1,\cdots,v_n]$ and $v_i=g^{-1}(u_i\otimes1+w_i\otimes\eta)$ for every $i$, we have $\Phi(A,\sigma)\subseteq g^{-1}(S\otimes_FK)$.
On the other hand $\Phi(A,\sigma)$ is maximal commutative and $g^{-1}(S\otimes_FK)\subseteq A$ is commutative.
Hence $\Phi(A,\sigma)=g^{-1}(S\otimes_FK)$.
We claim that \[\mathfrak{Pf}(A,\sigma)(g^{-1}(s_1\otimes1),g^{-1}(s_2\otimes1))\in F \quad {\rm for}\quad s_1,s_2\in S.\]
Write $\mathfrak{Pf}(A,\sigma)(g^{-1}(s_1\otimes1),g^{-1}(s_2\otimes1))=a+b\eta$ for some $a,b\in F$.
Then
\begin{equation*}
g^{-1}(s_1s_1\otimes1)+a+b\eta=g^{-1}(s_1\otimes1)g^{-1}(s_2\otimes1)+a+b\eta\in \alt(A,\sigma).
\end{equation*}
The equality $g(\alt(A,\sigma))=\alt(B,\sigma')\otimes_FK$ then implies that
$(s_1s_2+a)\otimes1+b\otimes\eta\in\alt(B,\sigma')\otimes_FK$, so
$b\in\alt(B,\sigma')$.
Since $\sigma'$ is orthogonal, we have $1\notin\alt(B,\sigma')$, hence $b=0$.
This proves the claim.
Finally, define the bilinear form $\mathfrak{b}:S\times S\rightarrow F$ via \[\mathfrak{b}(s_1,s_2)=\mathfrak{Pf}(A,\sigma)(g^{-1}(s_1\otimes1),g^{-1}(s_2\otimes1)).\]
Then the restriction of $g^{-1}$ to $S\otimes_FK$ defines an isometry $(S\otimes_FK,\mathfrak{b}_K)\iso(\Phi(A,\sigma),\mathfrak{Pf}(A,\sigma))$.
\end{proof}

\section{Separable descent in general case}\label{sec-sep}
We continue to assume that $K=F(\eta)$, where $\eta^2+\eta=\delta\in F\setminus\wp(F)$.

\begin{lem}\label{norm}
Let $E=K(\sqrt\alpha)$ and $L=F(\sqrt\alpha)$ for some $\alpha\in F\setminus K^2$.
Let $A$ be a central simple algebra over $K$.
If $\co_{K/F}(A)$ splits, then $\co_{E/L}(A\otimes_KE)$ is also split.
\end{lem}

\begin{proof}
If $\co_{K/F}(A)$ splits, then by \cite[(3.1 (2))]{knus} there exists an involution $\sigma$ of the second kind on $A$ with $\sigma|_F=\id$.
Note that $E=L(\eta)$ is a separable quadratic extension of $L$.
Let $\iota$ be the nontrivial automorphism of $E/L$.
Then the map $\sigma\otimes\iota$ is an involution of the second kind on $A\otimes_KE$ which leaves $L$ elementwise invariant.
Hence $\co_{E/L}(A\otimes_KE)$ splits by \cite[(3.1 (2))]{knus}.
\end{proof}

\begin{lem}\label{splits}
Let $(A,\sigma)$ be a central simple algebra with involution of the first kind over $K$.
If $(A,\sigma)$ has a descent to $F$, then $\co_{K/F}(A)$ splits.
\end{lem}

\begin{proof}
The result follows from \cite[(3.13 (5))]{knus} and \cite[(3.1 (1))]{knus}.
\end{proof}

\begin{thm}\label{main}
For a totally decomposable algebra with orthogonal involution $(A,\sigma)$ over $K$ the following conditions are equivalent.
\begin{itemize}
  \item [(1)] $(A,\sigma)$ has a descent to $F$.
  \item [(2)] $(A,\sigma)$ has a totally decomposable descent to $F$.
  \item [(3)] $\co_{K/F}(A)$ splits and $\mathfrak{Pf}(A,\sigma)\simeq\mathfrak{b}_K$ for some symmetric bilinear form $\mathfrak{b}$ over~$F$.
\end{itemize}
Furthermore, if $\sigma$ is isotropic, the above conditions are equivalent to:
\begin{itemize}
  \item [(4)] $\co_{K/F}(A)$ splits and $\Phi(A,\sigma)\simeq S\otimes_FK$ for some totally singular conic $F$-algebra $S$.
\end{itemize}
\end{thm}

\begin{proof}
The implication $(2)\Rightarrow(1)$ is evident and $(1)\Rightarrow(3)$ follows from (\ref{splits}) and (\ref{ani}).

$(3)\Rightarrow(2)$:
If $\mathfrak{Pf}(A,\sigma)\simeq\mathfrak{b}_K$ for some symmetric bilinear form $\mathfrak{b}$ over $F$, then by \cite[(3.3)]{dolphin}, $\mathfrak{b}$ is similar to a Pfister form $\mathfrak{c}$ over $F$.
Hence $\mathfrak{c}_K$ is similar to $\mathfrak{Pf}(A,\sigma)$.
Since $\mathfrak{c}_K$ and $\mathfrak{Pf}(A,\sigma)$ are both Pfister forms, we get $\mathfrak{c}_K\simeq\mathfrak{Pf}(A,\sigma)$ by \cite[(2.4)]{dolphin}.
Write $\mathfrak{c}=\lla\alpha_1,\cdots,\alpha_n\rra$ for some $\alpha_1,\cdots,\alpha_n\in F^\times$.
Then $\mathfrak{Pf}(A,\sigma)\simeq\lla\alpha_1,\cdots,\alpha_n\rra_K$.

By \cite[(3.1)]{me} there exists a decomposition $(A,\sigma)\simeq\bigotimes_{i=1}^n(Q_i,\sigma_i)$ into quaternion algebras with orthogonal involution over $K$ such that $\disc\sigma_i=\alpha_iK^{\times2}\in\penalty 0 K^\times/K^{\times2}$, $i=1,\cdots,n$.
If $\alpha_i\in F^{\times2}$ for every $i$, then (\ref{disc}) shows that
\[\textstyle(A,\sigma)\simeq_K\bigotimes_{i=1}^n(M_2(K),t)\simeq_K\bigotimes_{i=1}^n(M_2(F),t)\otimes_F(K,\id),\]
and we are done.
Otherwise, (by re-indexing) we may assume that $\alpha_n\notin F^{\times2}$.
Choose $v_i\in\alt(Q_i,\sigma_i)$ such that $v_i^2=\alpha_i\in F^\times$, $i=1,\cdots,n$.
By induction on $n$ we prove a stronger result than we need, namely that for $i=1,\cdots,n$, there exists a quaternion algebra with involution $(M_i,\tau_i)$ over $F$ such that
\[\textstyle(A,\sigma)\simeq_K\bigotimes_{i=1}^n(M_i,\tau_i)\otimes_F(K,\id),\]
and $v_i\in \alt(M_i,\tau_i)$.
The case $n=1$ follows from (\ref{quaternion}), hence let $n\geq2$.
Set $E=K[v_n]$ and $L=F[v_n]=F(\sqrt{\alpha_n})$.
Since $K^{\times2}\cap F^\times=F^{\times2}$, we have $\alpha_n\notin K^{\times2}$, so $E=K(\sqrt{\alpha_n})$.
Let $B=C_A(v_n)$ be the centralizer of $v_n$ in $A$ (we have identified $v_n\in Q_n$ with an element of $A$).
Then
\[\textstyle(B,\sigma|_B)\simeq_E\bigotimes_{i=1}^{n-1}(Q_i,\sigma_i)\otimes_K(E,\id),\]
is a totally decomposable $E$-algebra of degree $2^{n-1}$ with orthogonal involution.
By (\ref{norm}), $\co_{E/L}(A\otimes_KE)$ splits.
Since $B\sim A\otimes_KE$, by \cite[Ch. 8, (9.8)]{schar} we have $\co_{E/L}(B)\sim \co_{E/L}(A\otimes_KE)$,
hence $\co_{E/L}(B)$ is also split.
Note that $\mathfrak{Pf}(B,\sigma|_B)\simeq\lla\alpha_1,\cdots,\alpha_{n-1}\rra_E$ and $E=L(\eta)$ is a separable quadratic extension of $L$.
By identifying $v_i\otimes1\in\alt((Q_i,\sigma_i)\otimes_K(E,\id))$ with $v_i$, the induction hypothesis implies that for $i=1,\cdots,n-1$, there exists a quaternion algebra with orthogonal involution $(Q'_i,\sigma'_i)$ over $L$ such that
\[\textstyle(B,\sigma|_B)\simeq_E\bigotimes_{i=1}^{n-1}(Q'_i,\sigma'_i)\otimes_L(E,\id),\]
and $v_i\in \alt(Q'_i,\sigma'_i)$.
Since $L^2\subseteq F$, by (\ref{new}) there exists a quaternion $F$-algebra with orthogonal involution $(M_i,\tau_i)$ such that $(Q'_i,\sigma'_i)\simeq_L(M_i,\tau_i)\otimes_F(L,\id)$ and $v_i\in \alt(M_i,\tau_i)$, $i=1,\cdots,n-1$.
It follows that
\[\textstyle(B,\sigma|_B)\simeq_E\bigotimes_{i=1}^{n-1}(M_i,\tau_i)\otimes_F(L,\id)\otimes_L(E,\id)\simeq_E\bigotimes_{i=1}^{n-1}(M_i,\tau_i)\otimes_F(E,\id).\]
The $K$-algebra $\bigotimes_{i=1}^{n-1}M_i\otimes_FK$ may be identified with a subalgebra of $A\subseteq B$.
Set
\[\textstyle Q''_n=C_A(\bigotimes_{i=1}^{n-1}M_i\otimes_F K)\quad {\rm and}\quad \sigma''_n=\sigma|_{Q''_n}.\]
Then $(Q''_n,\sigma''_n)$ is a quaternion $K$-algebra with orthogonal involution and \[\textstyle(A,\sigma)\simeq_K\bigotimes_{i=1}^{n-1}(M_i,\tau_i)\otimes_F(K,\id)\otimes_K(Q''_n,\sigma''_n).\]
Note that according to (\ref{sc}), $\co_{K/F}(M_i\otimes_FK)$ splits for every $i$.
Since $\co_{K/F}(A)$ split, \cite[(3.13 (2))]{knus} implies that $\co_{K/F}(Q''_n)$ is also split.
As $v_n$ commutes with $B$ and $\bigotimes_{i=1}^{n-1}M_i\otimes_FK\subseteq B$, we have $v_n\in Q''_n$.
Using \cite[(3.5)]{mn}, we get $v_n\in\alt(Q''_n,\sigma''_n)$.
Hence (\ref{quaternion}) implies that $(Q''_n,\sigma''_n)\simeq_K(M_n,\tau_n)\otimes_F(K,\id)$
for some quaternion $F$-algebra with involution $(M_n,\tau_n)$ with $v_n\in\alt(M_n,\tau_n)$.
It follows that
\[\textstyle(A,\sigma)\simeq_K\bigotimes_{i=1}^{n}(M_i,\tau_i)\otimes_F(K,\id),\]
and $v_i\in\alt(M_i,\tau_i)$, $i=1,\cdots,n$.
This completes the proof of $(3)\Rightarrow(2)$.

Finally, the implication $(1)\Rightarrow(4)$ follows from (\ref{splits}) and (\ref{an}), even without the isotropy condition on $\sigma$.
If $\sigma$ is isotropic, the implication $(4)\Rightarrow(3)$ follows from (\ref{b}).
\end{proof}

We conclude by an example which shows that if $\sigma$ is anisotropic, the implication $(4)\Rightarrow(1)$ in (\ref{main}) does not hold.

\begin{exm}
Suppose that $F^2\neq F$.
Let $\lambda\in F\setminus F^2$ and set $\beta=\lambda+\delta+\eta\in K$.
Consider the involution $\sigma:M_2(K)\rightarrow M_2(K)$ defined by
\begin{align*}
\sigma\left(\begin{array}{cc}a & b \\c & d\end{array}\right)=\left(\begin{array}{cc}a & c\beta^{-1} \\b\beta & d\end{array}\right).
\end{align*}
Set
\begin{align*}
v=\left(\begin{array}{cc}0 & 1 \\\beta & 0\end{array}\right)\quad {\rm and}\quad w=\left(\begin{array}{cc}\eta & 1 \\\beta & \eta\end{array}\right).
\end{align*}
Then $v\in \alt(M_2(K),\sigma)$ and $w\in\sym(M_2(K),\sigma)$.
We also have $v^2=\beta$ and $w=v+\eta$ (we have identified scalar matrices in $M_2(K)$ with elements of $K$).
In particular, $\disc\sigma=\beta K^\times/K^{\times2}$.
Note that $\co_{K/F}(M_2(K))$ splits and $\Phi(A,\sigma)=K[v]=K[w]\simeq F[w]\otimes_FK$.
Since $w^2=\lambda\in F$, $F[w]$ is a totally singular conic $F$-algebra.
Hence, the pair $(M_2(K),\sigma)$ satisfies the condition $(4)$ in (\ref{main}).
If $(M_2(K),\sigma)$ has a descent to $F$, then by (\ref{quaternion}), there exists $u\in\alt(M_2(K),\sigma)$ such that $u^2\in F^\times$.
As $\alt(M_2(K),\sigma)$ is one-dimensional, we have $\alt(M_2(K),\sigma)=Kv$.
Write $u=(b+c\eta)v$ for some $b,c\in F$.
Then
\[u^2=(\lambda b^2+\delta b^2+\lambda \delta c^2+c^2\delta^2+\delta c^2)+(b^2+c^2+\lambda c^2)\eta.\]
Since $u^2\in F^\times$, we have $b^2+c^2+\lambda c^2=0$.
If $c=0$, then $b=0$ which implies that $u=0$, contradicting $u^2\in F^\times$.
If $c\neq0$, then $\lambda =b^2c^{-2}+1\in F^2$, which is again a contradiction.
Hence, $(M_2(K),\sigma)$ has no descent to $F$, i.e., the implication $(4)\Rightarrow(1)$ in (\ref{main}) does not hold for $(M_2(K),\sigma)$.
\end{exm}

\scriptsize{
\noindent
A.-H. Nokhodkar, {\tt
  a.nokhodkar@kashanu.ac.ir},\\
Department of Pure Mathematics, Faculty of Science, University of Kashan, P.~O. Box 87317-51167, Kashan, Iran.}

\end{document}